\documentclass[11pt]{article}
\usepackage{graphicx,tikz-cd}
\graphicspath{ {./images/} }
\usepackage[utf8]{inputenc}
\usepackage[english]{babel}
\usepackage{enumerate}
\usepackage{amsthm} 
\usepackage{amsmath}
\usepackage{amssymb} 
\usepackage{soul}
\usepackage{amsmath}
\usepackage{cancel}
\usepackage{listings}
\usepackage{xcolor}
\usepackage[margin=1in]{geometry}
\usepackage[normalem]{ulem}
\usepackage{csquotes} 

\definecolor{codegreen}{rgb}{0,0.6,0}
\definecolor{codegray}{rgb}{0.5,0.5,0.5}
\definecolor{codepurple}{rgb}{0.58,0,0.82}
\definecolor{backcolour}{rgb}{0.95,0.95,0.92}

\lstdefinestyle{mystyle}{
    backgroundcolor=\color{backcolour},   
    commentstyle=\color{codegreen},
    keywordstyle=\color{magenta},
    numberstyle=\tiny\color{codegray},
    stringstyle=\color{codepurple},
    basicstyle=\ttfamily\footnotesize,
    breakatwhitespace=false,         
    breaklines=true,                 
    captionpos=b,                    
    keepspaces=true,                 
    numbers=left,                    
    numbersep=5pt,                  
    showspaces=false,                
    showstringspaces=false,
    showtabs=false,                  
    tabsize=2
}

\usepackage{authblk}
\usepackage{makecell} 
\usepackage{chemfig} 
\usepackage{blindtext}
\usepackage{hyperref}
\usepackage{appendix}
\hypersetup{
    colorlinks=true,
    linkcolor=blue,
    filecolor=magenta,      
    urlcolor=cyan,
    citecolor=blue
    }
\usepackage{array}
\usepackage{caption} 
\usepackage{titlesec} 

\usepackage[none]{hyphenat} 

\makeatletter
\let\@fnsymbol\@arabic
\makeatother

\newcommand{\change}[3]{\Delta#1(#2,#3)}

\theoremstyle{definition}
\newtheorem{theorem}{Theorem}[section]
\newtheorem{thm}[theorem]{Theorem}
\newtheorem{lemma}[theorem]{Lemma}

\newtheorem{defn}[theorem]{Definition}

\newtheorem{ex}[theorem]{Example}

\newtheorem{remark}[theorem]{Remark}
\newtheorem{notation}[theorem]{Notation}

\newcommand{\N}{\mathbb N}
\newcommand{\Z}{\mathbb Z}

\newcommand{\RR}{\mathcal R}
\newcommand{\CC}{\mathcal C}
\newcommand{\XX}{\mathcal X}
\newcommand{\NN}{\mathcal N}

\usepackage{nicematrix}

\definecolor{TAMU}{RGB}{140,0,0}
\definecolor{myblue}{RGB}{0,0,198}
\definecolor{myred}{RGB}{182,0,0}

\title{
Deficiency of chemical reaction networks: \\
The effect of operations that preserve multistationarity \\ and periodic orbits
}
\author{Awildo Gutierrez\thanks{Department of Mathematics and Statistics, Hamilton College. awildo.gutierrez(at)wisc.edu}, Elijah Leake\thanks{Department of Mathematical Sciences, DePaul University. ejohnj247(at)gmail.com}, Caelyn Rivas-Sobie\thanks{Department of Mathematics and Natural Sciences, Arizona State University - West. csobie1(at)jh.edu}, \\Jordy Lopez Garcia\thanks{Department of Mathematics, Texas A\&M University. jordy.lopez(at)tamu.edu}, Anne Shiu\thanks{Department of Mathematics, Texas A\&M University. annejls(at)tamu.edu}}

\date{November 29, 2024}

\begin{document}
\maketitle
\begin{abstract}
We investigate six operations on chemical reaction networks, all of which have been proven to preserve important dynamical properties, namely, the capacity for 
nondegenerate multistationarity (multiple steady states) and periodic orbits.  
Both multistationarity and periodic orbits are properties that are known to be precluded when the deficiency (a nonnegative integer associated to a network) is zero.  
It is therefore natural to conjecture that the deficiency never decreases when any of the six aforementioned network operations are performed. 
We prove that this is indeed the case, and moreover, we characterize the numerical difference in deficiency after performing each network operation.

\vskip 0.1in

\noindent
{\bf Keywords:}
deficiency, reaction network, cyclomatic number

\noindent
{\bf MSC Codes:}
37N25,  
05C90, 
05C76, 
15A03 
\end{abstract}

\section{Introduction}
The aim of this work is to make a connection between two important research streams pertaining to dynamical systems arising from chemical reaction networks: (1) deficiency theory, which dates back to the work of Feinberg, Horn, and Jackson in the 1970s~\cite{Feinberg,Horn, horn_jackson_1972}; and (2) a collection of results that ``lift'' dynamical properties from smaller networks to larger ones~\cite{oscillatory, Banaji,species, nondegenerate,add-flow,ME_entrapped,Atoms}.  This second stream is more recent than the first, having begun only in the mid-2000s.  

To give more details, deficiency theory uses the ``deficiency'' of a network -- a nonnegative integer associated to a network which is computed by graph-theoretic and linear-algebraic means -- to serve as a link between the network structure and the dynamics arising from a network (via mass-action kinetics).  Roughly speaking, networks with low deficiency have simple dynamics.  The Deficiency Zero Theorem, for instance, rules out multistationarity (multiple positive steady states) and periodic orbits in networks with zero deficiency.  Consequently, many recent results 
have focused on recognizing general mass-action ODE systems as being ``dynamically equivalent'' to those arising from low-deficiency networks~\cite{craciun-uniqueness, haque-disguised, hernandez-framework,johnston-compute-w-rev}.

In contrast, the second research stream is interested in behavior that we noted above never occurs for deficiency-zero networks -- namely, nondegenerate multistationarity and periodic orbits (which are significant for networks arising in systems biology).  Indeed, so-called ``lifting'' (or ``inheritance'') results assert that such behaviors are preserved when chemical reaction networks are enlarged in certain ways.  

Such network operations include adding new reactions and/or new chemical species (under certain hypotheses).  Six of these operations, which we label E1--E6 (following Banaji), are described in Table~\ref{tab:summary-of-results} below, and (as summarized by Banaji) they represent essentially the best known ``lifting'' results~\cite{Banaji,banaji2023inheritance} (see Remark~\ref{rem:cappelletti} later in our work for a discussion of an additional operation).

To summarize, operations like E1--E6 allow nondegenerate multistationarity and periodic orbits (and, in fact, bifurcations~\cite{banaji2023inheritance})
to be ``lifted'', while deficiency-zero networks never exhibit such behaviors.  It is thus natural to conjecture that network operations that preserve nondegenerate behaviors never result in decreased  deficiency.  Our main result, as follows, asserts that this conjecture holds for the operations E1--E6, and moreover establishes the first link (to our knowledge) between the research streams of deficiency theory and ``lifting'' results:

\begin{thm}[Main result] \label{thm:main-result-summary}
    If $\NN'$ is a chemical reaction network obtained from a network $\NN$ by one of the operations E1--E6,
    then 
    \[ \delta({\NN'})\geq \delta({\NN})~,\]
    where $\delta(\NN)$ denotes the deficiency of $\NN$, and 
    the difference $\delta({\NN'}) - \delta({\NN})$ is given in Table~\ref{tab:summary-of-results}.
\end{thm}

\begin{center}
    
\begin{table}[ht]

\begin{tabular}{c l c}
\hline
Operation & Description &  
    $\delta(\NN')- \delta(\NN)$ \\
\hline
\hyperlink{E1_def}{E1} & 
\makecell[l]{Add a new linearly dependent reaction} & 
\hyperlink{E1_thm}{$0$ or $1$} \\
\hyperlink{E2_def}{E2} & 
\makecell[l]{Add reactions $0 \leftrightarrows X_i$ for all species $X_i$} &

\hyperlink{E2_thm}{$\operatorname{rk}(\NN) + m_\NN + \widetilde \ell_\NN - s_\NN - 1$ }\\
\hyperlink{E3_def}{E3} & 
\makecell[l]{Add a new linearly dependent species} &
\hyperlink{E3_thm}{$\operatorname{cyc}(\NN) - \operatorname{cyc}(\NN')$} \\
\hyperlink{E4_def}{E4} & 
\makecell[l]{Add a new species $Y$ and the pair of \\ \quad reversible reactions $0 \leftrightarrows Y$}
& \hyperlink{E4_thm}{$\operatorname{cyc}(\NN) - \operatorname{cyc}(\NN') + 1$} \\

\hyperlink{E5_def}{E5} & 
\makecell[l]{Add reversible reactions with new species \\ \quad such that rank condition holds} &
\hyperlink{E5_thm}{$0$}\\

\hyperlink{E6_def}{E6} & 
\makecell[l]{Split reactions and add complexes involving \\ \quad new species with rank condition} &
\hyperlink{E6_thm}{$0$}\\
\hline

\end{tabular}
\caption{
Summary of results.  
Here, $\delta(\NN)$,
$\operatorname{rk}(\NN)$, and $s_\NN$ denote, respectively, the deficiency, rank, and number of species of a network $\NN$.
Also, $\widetilde \ell_{\NN}$ is the number of linkage classes in $\NN$ that contain 
an at-most-unimolecular complex (namely, $0$ or some $X_i$), while $m_{\NN}$ is the number of 
at-most-unimolecular complexes that are missing from $\NN$ (see Notation \ref{notation:special-complex}). 
Finally, $\operatorname{cyc}(\NN)$ is the cyclomatic number of $\NN$ (see Definition \ref{def:cyclo-deficiency}), and the rank condition is given in Definition~\ref{def:rank-condition}. 
For details, see Theorems~\ref{thm:E1},~\ref{thm:E2},~\ref{thm:E3},~\ref{thm:E4},~\ref{thm:E5},
and~\ref{thm:E6}.
\label{tab:summary-of-results}}
\end{table}
\end{center}
\vspace{-1cm}

The outline of our work is as follows. In Section~\ref{sec:background}, we define chemical reaction networks, deficiency, and the network operations E1--E6.  Our results are proven in Section~\ref{sec:results}, and we end with a discussion in Section~\ref{sec:discussion}.

\section{Background} \label{sec:background}
This section introduces chemical reaction networks and the concepts of deficiency and cyclomatic number (Section~\ref{sec:network}) and recalls Banaji's network operations E1--E6 (Section~\ref{sec:operations}).

\begin{notation}
 We let $\N$ denote the set of positive integers, and $\N_0$ the set of nonnegative integers. For $n \in \N$, we define the set $[n]:=\{1,2,\dots,n\}$.
 \end{notation}

\subsection{Chemical reaction networks} \label{sec:network}

\begin{defn}\label{def:chemical-reaction}
    A {\em chemical reaction network} (or {\em network}, for short) $\NN$ is a triple $(\XX_\NN, \CC_\NN,\RR_\NN)$ consisting of three finite, nonempty sets:
    \begin{itemize}
        \item a set of \textit{species} $\XX_\NN$ = $\{X_1, X_2,..., X_{s_\NN}\}$, where $s_{\NN}\in \N$,
        \item a set of {\em complexes} $\CC_\NN$ = $\{C_1, C_2, ..., C_{c_\NN}\}$, where $c_\NN \in \N$ and each $C_{i}$ is an $\N_0$-linear combination of species in $\XX_\NN$, and
        \item a set of \textit{reactions} $\RR_\NN
        \subseteq 
        \{C_i \to C_j \mid i,j \in [c_\NN] \text{ with } i \ne j\}$, such that each complex $C_k \in \CC_\NN$ takes part in at least one reaction in $\RR_\NN$.
    \end{itemize}
    Let  $r_{\NN} := |\RR_\NN|$ denote the number of reactions of $\NN$.  
\end{defn}

If $C_i \to C_j$ and $C_j \to C_i$ are both in $\RR_\NN$, then these reactions are \textit{reversible}
and the pair is denoted by $C_i \leftrightarrows C_j$. 
Reversible reactions of the form $0\leftrightarrows X_{i}$ are called \textit{inflow-outflow reactions}.

It is convenient to represent a network $\NN$
by a directed graph, 
where the vertices are complexes and the edges correspond to reactions.  
A connected component of this directed graph is a {\em linkage class}, and we let $\ell_\NN \in \N$ denote the number of linkage classes.

For convenience, we list some basic parameters of a network in Table~\ref{tab:notation}.

\begin{center}
    
\begin{table}[ht]

\begin{tabular}{c c c c c c c}
\hline
$s_\NN$ & $c_\NN$& $r_\NN$& $\ell_\NN$& $\operatorname{rk}(\NN)$ 
\\
\hline
\#  species &
\# complexes &
\# reactions & 
\# linkage classes &
rank 
\\
\hline
\end{tabular}
\caption{Parameters associated to a network $\NN$.  \label{tab:notation}}
\end{table}
\end{center}
\vspace{-1cm}
For ease of notation, our examples use $A,B,C, \dots$ for species, rather than $X_1,X_2,X_3,\dots$.

\begin{ex}[Running example] \label{ex:complexes} 
    The following directed graph represents the network $\NN = (\XX_\NN,\CC_\NN,\RR_\NN)$, where ${\XX_\NN = \{A,B\}}$, $\CC_\NN=\{2A,A+2B,2A+2B\}$, and $\RR_\NN = \{2A\to A+2B,~ A+2B \to 2A + 2B\}$:  
    \begin{align*} 
        2A \longrightarrow A+2B \longrightarrow 2A + 2B~.
    \end{align*}
This network $\NN$ has $\ell_\NN =1$ linkage classes.  This network is our running example for this section.
\end{ex}

\begin{defn} \label{def:cpx-vec} Consider a network $\NN = (\XX_\NN, \CC_\NN,\RR_\NN)$. 
\begin{enumerate}
    \item For a complex $C_i = \alpha_1 X_1 + \alpha_2 X_2 + \dots + \alpha_{s_\NN} X_{s_\NN}$, where $\alpha_i \in \N_0$ for all $i \in [s_\NN]$, 
    its {\em complex vector} is $(\alpha_1,\alpha_2, \dots, \alpha_{s_\NN})^{\intercal } \in \N_0^{s_\NN}$.  
    In a slight abuse of notation, we sometimes let $C_i$ denote its complex vector.
    \item The {\em reaction vector} of a reaction $C_i \to C_j \in \RR_\NN$ is $C_j - C_i\in \Z^{s_\NN}$.
    \item Given an ordering of the reactions $\RR_\NN = \{ R_1,R_2,\dots, R_{r_\NN} \}$, the {\em stoichiometric matrix} of $\NN$, denoted by  $\Gamma_\NN$, is the $(s_\NN \times r_\NN)$-matrix in which the $i$-th column is the reaction vector of $R_i$.
    \item The {\em stoichiometric subspace} of $\NN$ is the column space of $\Gamma_\NN$ (or, equivalently, the 
    span of the reaction vectors arising from all reactions of $\NN$), and the dimension of this subspace is the {\em rank} of $\NN$, which we denote by $\operatorname{rk}(\NN)$.
\end{enumerate}
\end{defn} 

\begin{ex}[Example~\ref{ex:complexes}, continued] \label{ex:running-example-continued}
For our running example, 
$\NN = \{A \to A+2B \to 2A + 2B\}$,
the complex vectors and the stoichiometric matrix are, respectively:
    $$\begin{bmatrix} 1 \\ 0 \end{bmatrix},
    \begin{bmatrix}1 \\ 2 \end{bmatrix},
    \begin{bmatrix}2 \\ 2\end{bmatrix}~,
    \quad \quad {\rm and} \quad \quad
    \Gamma_\NN =
    \begin{bmatrix}
        0 & 1 \\
        2 & 0
    \end{bmatrix}.$$
The two columns of $\Gamma_\NN $ are linearly independent, so 
$\operatorname{rk}(\NN)=2$.
\end{ex}

The following result is well known, and we include a proof for completeness.

\begin{lemma}\label{E5' cycles lemma}
    If complexes $C$
    and $D$ are in the same linkage class of a network $\NN$, then $D-C$ is in the stoichiometric subspace of $\NN$.
\end{lemma}

\begin{proof}
Assume that $C$ and $D$ are in the same linkage class of $\NN$. Then there exists a path of complexes $C_1 = C, ~C_2, ~C_3,\dots, C_q= D$, such that, for all $j\in[q-1]$, at least one of the reactions $C_j \to C_{j+1}$ and $C_j \leftarrow C_{j+1}$ is a reaction of $\NN$.
Hence, for all $j\in[q-1]$, the vector $C_{j+1} - C_{j}$ or its negative is a reaction vector of $\NN$, and so
 $C_{j+1} - C_{j}$
is in the stoichiometric subspace of $\NN$.  Hence, the sum 
$\sum_{j=1}^{q-1} \left(C_{j+1} - C_{j} \right) = D-C$
is in the stoichiometric subspace of $\NN$.
\end{proof}

\begin{defn}\label{def:cyclo-deficiency}
    The \textit{cyclomatic number} and  
    \textit{deficiency} of a network $\NN$ are, respectively, 
    $$\operatorname{cyc}(\NN) ~:=~ r_\NN - c_\NN + \ell_\NN
    \quad \quad  {\rm and } \quad \quad 
	\delta(\NN) ~:=~ c_\NN - \ell_\NN - \operatorname{rk}(\NN),$$
where $r_{\NN}$, $c_\NN$, $\ell_\NN$, and $\operatorname{rk}(\NN)$ are as in Table~\ref{tab:notation}.
\end{defn}

\begin{remark}[Cyclomatic numbers] \label{rem:cyclomatic} 
The cyclomatic number of a network $\NN$ 
comes from graph theory: 
it is the so-called cyclomatic number of the directed graph associated to $\NN$.  
This cyclomatic number of a graph $G=(V,E)$, also called the circuit rank, 
represents the number of independent  cycles in $G$ or (equivalently) the minimum number of edges that must be removed to break all cycles~\cite[Chapter~4]{Berge2001}.  The cyclomatic number also captures the dimension of the first homology group: $H_{1}\left(G\right)\cong \Z^{|E|-|V|+|C|}$, where $C$ is the set of connected components of $G$.  
In our setting, $G$ is the graph associated to $\NN$, so we identify 
$|E|, |V|$, and $|C|$, with 
$r_{\NN}, c_{\NN}$, and 
$\ell_{N}$, respectively.
\end{remark}

\begin{ex}[Example~\ref{ex:running-example-continued}, continued] \label{ex:cyclomatic}
    For our running example, 
    $\NN = \{A \to A+2B \to 2A + 2B\}$,
the cyclomatic number  is  $\operatorname{cyc}(\NN) = 2 - 3 + 1 = 0$, which is consistent with Remark~\ref{rem:cyclomatic}, as $\NN$ has no cycles.  
    The deficiency of $\NN$ is $\delta(\NN) = 3 - 1 - 2 = 0$.
\end{ex}

\subsection{Network operations}
\label{sec:operations}

In this subsection, we recall the network operations\footnote{Banaji calls these operations ``enlargements'', which is why they are labeled $E1$--$E6$.} investigated by Banaji~\cite{Banaji} (Definition~\ref{def:ops}).  

\begin{defn}[Adding a new species, splitting reactions] \label{def:split}
Let $\NN$ be a network.  
\begin{enumerate}
    \item Consider a reaction of $\NN$: 
    \begin{align} \label{eq:reaction-before-add-species}
            a_1X_1 + a_2X_2 + \dots + a_{s_\NN}X_{s_\NN} ~ \longrightarrow ~b_1X_1 + b_2X_2 + \dots + b_{s_\NN}X_{s_\NN}~.
    \end{align}
    By replacing the reaction~\eqref{eq:reaction-before-add-species} by one of the following form:
    \begin{align*} 
    a_1X_1 + a_2X_2 + \dots + a_{s_\NN}X_{s_\NN} + c_1Y ~ \longrightarrow ~ b_1X_1 + b_2X_2 + \dots + b_{s_\NN}X_{s_\NN} + c_2Y~,
    \end{align*}
    for some $c_1, c_2 \in \N_0$ -- where at least one of $c_1,c_2$ is nonzero --
    we obtain the network $\NN'$ from $\NN$ by {\em adding a new species} $Y$ to reaction~\eqref{eq:reaction-before-add-species}.
    \item 
    Let $C_i \to C_j$ be a reaction of $\NN$.  The network $\NN'$ obtained from $\NN$ by {\em splitting} the reaction $C_i \to C_j$ 
    arises from replacing 
    $C_i \to C_j$ with the reactions $C_i \to D$ and $D \to C_j$, where $D$ is a new complex (that is, with $D \notin \CC_\NN$).  

\end{enumerate}    
\end{defn}

\begin{defn}[Rank condition] 
\label{def:rank-condition}
Let $m \in \mathbb{N}$. Consider a network $\NN'$ that is obtained in some way from a network $\NN$ and involves at least $m$ new species (here, new species are species that are not in $\XX_\NN$).  
Let $\widetilde \Gamma_{\NN'}$ denote the submatrix of $\Gamma_{\NN'}$ formed by the row(s) corresponding to the new species.  The {\em rank-$m$ condition} is satisfied if $\widetilde \Gamma_{\NN'}$ has rank $m$.
\end{defn}

\begin{defn}[Operations] \label{def:ops} 
Denote by $\NN'$ the network obtained from a network $\NN$ after one of the following {\em operations}. 
\begin{enumerate}
    \item[\hypertarget{E1_def}{E1.}] Add to $\RR_\NN$ a new reaction involving only species in $\XX_\NN$ such that $\operatorname{rk}(\NN) =\operatorname{rk}(\NN')$.
    
    \item[\hypertarget{E2_def}{E2.}] Add to $\RR_\NN$ the inflow-outflow reactions for all species in $\XX_\NN$.
    
    \item[\hypertarget{E3_def}{E3.}] Add a new species $Y$ into some or all of the existing reactions of $\NN$ such that $\operatorname{rk}(\NN) =\operatorname{rk}(\NN')$. 
    
    \item[\hypertarget{E4_def}{E4.}] Add a new species $Y$ into some or all of the existing reactions of $\NN$, and add the inflow-outflow reactions $0\leftrightarrows Y$.
    
    \item[\hypertarget{E5_def}{E5.}] Add $m$ new pairs of reversible reactions to $\NN$ and $m+i$ new species only in the new reactions, for some $m \in \N$, $i \in \N_0$, such that the rank-$m$ condition holds (as in Definition~\ref{def:rank-condition}).
    
    \item[\hypertarget{E6_def}{E6.}] Split $m$ existing reactions of $\NN$, 
    such that the new complexes added involve $m+i$ new species, for some $m  \in \N$, $i \in \N_0$, 
    such that the rank-$m$ condition holds (as in Definition~\ref{def:rank-condition}).
\end{enumerate}
\end{defn}

\begin{remark}[E5] \label{rem:reversible-reactions-stoic-matrix}
The network operation E5 consists of adding $m$ pairs of reversible reactions. In this case, the stoichiometric matrix of $\NN'$ has the following block upper-triangular form:
\begin{align} \label{eq:rank-condition}
\Gamma_{\NN'}
~=~
\left(
	\begin{array}{ccc|ccc}
	&&&    * & \dots & * \\	
	& \Gamma_{\NN} & & * &  & *   \\
	&& & \vdots & & \vdots \\	
    \cline{1-3} 
	0 & \dots &  0 &  \vdots &    &\vdots  \\
	\vdots & \ddots &  \vdots &	* &  & * \\
    0 & \dots & 0 & * & \dots & * \\
	\end{array}
\right)~,
\end{align}
where the last rows correspond to the new species and the last $2m$ columns, indicated by $*$'s, correspond to the $2m$ new reactions (which come in $m$ pairs).  
For each pair, the corresponding two reaction vectors are negatives of each other, so it suffices (for the rank condition) to consider a version of the stoichiometric matrix $\Gamma_{\NN'}$ in which only one reaction vector is added for each pair. 
\end{remark}

\begin{remark} \label{rem:preserve-nondeg}
    As mentioned in the Introduction, 
    all six operations E1--E6 have been proven to preserve nondegenerate multistationarity and periodic orbits~\cite{Banaji}, as well as bifurcations~\cite{banaji2023inheritance}.
\end{remark}
 
\begin{ex}[E1; Example~\ref{ex:cyclomatic}, continued] \label{ex:E1-deficiency}
Recall that the network $\{      2A \rightarrow A+2B \rightarrow 2A + 2B\}$
has rank $2$.  Adding any reaction in the species $A$ and $B$ preserves the rank and so is an E1 operation.  
For instance, adding the reaction $2A + 2B \to 3A + 3B$ yields:
    \begin{align*}
            \NN' ~=~
        \{
        2A \rightarrow A+2B \rightarrow 2A + 2B  ~{\color{blue} \rightarrow 3A + 3B}
        \}~.
    \end{align*} 
Then $\delta (\NN') = 4-1-2=1$.
\end{ex}

\begin{ex}[E2] \label{ex:E2}
Consider the network $\NN = \{A \to 2C,~2D \leftarrow C \to B \}$.  The operation E2 adds inflow-outflow reactions, which results in the following network, which we denote by $\NN'$:

\begin{center}
\begin{tikzcd}
    & \textcolor{blue}{D} \arrow[r,shift right,blue]  & \textcolor{blue}{0} \arrow[d,shift right,blue]\arrow[dl,shift right=0.5ex,blue]\arrow[l,shift right,blue] \arrow[r,shift right,blue] & \arrow[l,shift right,blue] A \arrow[r] & 2C \\
    2D & \arrow[l] C\arrow[ur,shift right=0.6ex,blue]\arrow[r] & B\arrow[u,shift right,blue] & &
\end{tikzcd}
\end{center}
The deficiencies are $\delta(\NN)=5-2-3=0$ and $\delta(\NN')=7-1-4=2$.
\end{ex}

\begin{ex}[E3] \label{ex:E3} 
The following shows a network $\NN$ and a network $\NN'$ obtained after an E3 operation, in which a new species $D$ is added to two of the reactions:
    \begin{center}
        \includegraphics[width=12cm]{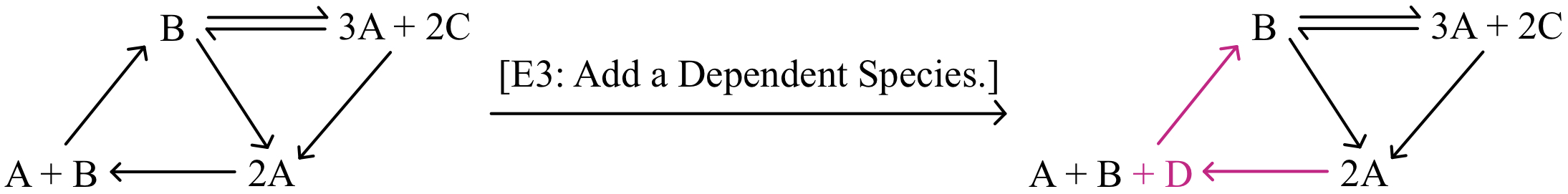}
    \end{center}
It is straightforward to check that both networks have rank $3$. The deficiencies and cyclomatic numbers are as follows:
$\delta(\NN) =\delta(\NN') = 4 - 1 - 3 = 0$ 
and 
$\operatorname{cyc}(\NN)=\operatorname{cyc}(\NN') = 6 - 4 + 1 = 3$.  
\end{ex}

\begin{ex}[E4] \label{ex:E4}
Consider the network $\NN=\{A \leftrightarrows B \}$, which has deficiency $\delta(\NN)=2-1-1=0$ and cyclomatic number $\operatorname{cyc}(\NN)=2-2+1=1$. 
We apply the E4 operation in which the species $C$ is added to the reaction $A\rightarrow B$ to obtain  $A \to B+C$.
The resulting network is $\NN'=\{
A \to B + {\color{blue} C},~
B \to A, ~ {\color{blue}0 \leftrightarrows C} \}$.  
The deficiency of $\NN' $ is 
$\delta(\NN')=5-2-2=1$, while the cyclomatic number is $\operatorname{cyc}(\NN')=4-5+2=1$.
\end{ex}

\begin{ex}[E5] \label{ex:E5}
Consider the network $\NN=\{A\to B,~ 2B \to 2A\}$.  
After applying an E5 operation in which two pairs of reversible reactions are added, we obtain the following network: 
    $\NN'=\{A\to B{\color{blue}~\leftrightarrows C+D},~ 2B \to 2A {\color{blue}~\leftrightarrows D+E} \}$.
A version of the stoichiometric matrix of $\NN'$ is as follows (where only one reaction for each pair is included, as in Remark \ref{rem:reversible-reactions-stoic-matrix}):
    $$\Gamma_{\NN'} ~=~ 
    \begin{bmatrix}
        -1 & 2 & 0 & -2 \\
        1 & -2 & -1 & 0  \\
        0 & 0 & 1 & 0\\
        0 & 0 & 1 & 1 \\
        0 & 0 & 0 & 1
    \end{bmatrix}.$$
The last three rows (corresponding to the new species $C,D,E$) form a rank-2 submatrix (as required by E5).  The deficiencies of the networks are $\delta(\NN)=4-2-1=1$ and $\delta(\NN')=6-2-3=1$.
\end{ex}

\begin{ex}[E6]  \label{ex:E6}
Consider the network $\NN = \{ 2A \to A+B \leftarrow 2B \}$.
We apply operation E6 in which both reactions are split and new complexes involving new species $C$ and $D$ are added:
\[
    \NN' ~=~ \{ 2A ~{\color{blue} \to A+C \to } ~ A+B ~{\color{blue} \leftarrow 2C+D \leftarrow } ~ 2B \}~.
\]

The stoichiometric matrices of these networks are as follows:

\begin{align*}
        \Gamma_{\NN} ~=~ \begin{bmatrix}
        -1 & 1 \\
        1 & -1
    \end{bmatrix}
\quad {\rm and}
\quad
        \Gamma_{\NN'} ~=~ \begin{bmatrix}
       -1 &  0 &  1 & 0  \\
        0 &  1 &  1 & -2  \\
        1 & -1 & -2 & 2  \\
        0 &  0 & -1 & 1 
    \end{bmatrix}~.
\end{align*}
As required for an E6 operation, the last two rows of $\Gamma_{\NN'}$ (corresponding to $C,D$) form a rank-2 submatrix.  
Notice also that 
$\operatorname{rk}(\NN) = 1$, $\operatorname{rk}(\NN') = 3$, 
$\delta(\NN) = 3 - 1 - 1 = 1$, and $\delta(\NN') = 5 - 1 - 3 = 1$. 
\end{ex}

\subsection{A more general rank condition} \label{sec:general-rank-condition}  
The operation E5 requires the rank-$m$ condition, that is, the rows of $\Gamma_{\NN'}$ corresponding to the new species must have rank $m$.  It is natural to consider the generalization of this condition in which we ask that the columns corresponding to new reactions (rather than the rows corresponding to new species) have rank $m$.  (The fact that this condition generalizes the original condition arises from the block triangular structure of $\Gamma_{\NN'}$, in~\eqref{eq:rank-condition}.)  

This new condition defines a generalization of the operation E5, which we denote by E5$'$.  If a network $\NN'$ is obtained from a network $\NN$ by performing an E5$'$  operation, then $\delta(\NN') \geq \delta(\NN)$ (this follows from Lemma~\ref{lem:add-reactions} in the next section).  In this way, E5$'$ is similar to (what we prove in the next section about) E1--E6. 

However, unlike E1--E6~\cite{Banaji}, the operation E5$'$ has not been proven to preserve nondegenerate multistationarity or periodic orbits.  In fact, E5 does not, in general, preserve nondegenerate multistationarity.  The following example illustrates this fact.

\begin{ex} \label{ex:E5-prime}
The network  $\NN = \{A \to B,~ 2A+B \to 3A\} $ exhibits nondegenerate multistationarity~\cite{JoshiShiu}, 
but 
the network 
$\NN' = \{A \to B,~ 2A+B \to 3A,~ 
{\color{blue} C \leftrightarrows A+C} \}$, 
which is obtained from $\NN$ by an E5$'$ operation, does not (this is a straightforward computation). 
\end{ex}
\noindent
As for (nondegenerate) periodic orbits, we conjecture that E5$'$ similarly does not, in general, preserve such behavior.

\section{Results}  \label{sec:results}

This section contains our results for the operations E1--E6.  We make use of the following notation. 
\begin{notation}[Change in a parameter] \label{notation:change}
    For  networks $\NN$ and $\mathcal M$, and parameter $\theta$ of networks, define
    $$\Delta \theta(\NN, \mathcal M) ~:=~ \theta( \mathcal M ) - \theta (\NN).$$
    The relevant parameters defined so far are $s_\NN$, $c_\NN$, $r_\NN$, $\ell_\NN$, $\operatorname{rk}(\NN)$, $\operatorname{cyc}(\NN)$, and $\delta(\NN)$. 
\end{notation}

Using Notation~\ref{notation:change}, the 
 following formulas for the changes in deficiency and cyclomatic number are immediate from Definition~\ref{def:cyclo-deficiency}: 
    \begin{align*}
        \change{\delta}{\NN}{\mathcal{M}} &~=~
    \change{c}{\NN}{\mathcal{M}} - \change{\ell}{\NN}{\mathcal{M}} - \change{\operatorname{rk}}{\NN}{\mathcal{M}}~, \quad {\rm and}  \\
        \change{\operatorname{cyc}}{\NN}{\mathcal{M}} &~=~ 
        \change{r}{\NN}{\mathcal{M}} - \change{c}{\NN}{\mathcal{M}} + \change{\ell}{\NN}{\mathcal{M}}~.    
    \end{align*}
We use the above equalities frequently in our proofs below.  

\subsection{Three useful lemmas} \label{sec:3-lemmas}

The following result pertains to operations that add reactions.

\begin{lemma}[Adding reactions]\label{lem:add-reactions} 
Let $C$ and $D$ be complexes.
Consider a network $\NN'$ that is obtained from a network $\NN$ by adding the reaction $C \to D$ (where $C\to D$ is not a reaction of $\NN$) or the pair of reversible reactions $C \leftrightarrows D$ (where neither $C\to D$ nor $C \leftarrow D$ is a reaction of $\NN$). Then:
\begin{enumerate}
    \item $\change{\delta}{\NN}{\NN'}=0$ or $\change{\delta}{\NN}{\NN'}=1$, and
    \item If no linkage class of $\NN$ contains both 
    $C$ and $D$ (for instance, if at least one of $C$ and $D$ is not a complex of $\NN$),
    then $\change{c}{\NN}{\NN'}-\change{\ell}{\NN}{\NN'} = 1$.  
\end{enumerate}
\end{lemma}
\begin{proof} 
As one reaction or one reversible pair is added, the rank either is unchanged or increases by $1$ (that is, $\change{\operatorname{rk}}{\NN}{\NN'} \in \{0,1\}$).
    We consider the following cases:

\uline{Case 1}:
$C$ and $D$ are both new complexes ($C, D \notin \CC_\NN$).  In this case, $\change{c}{\NN}{\NN'}=2$, and the reaction $C \to D$ (or $C \leftrightarrows D$) creates a new linkage class (so, $\change{\ell}{\NN}{\NN'} = 1$). 
Hence, $\change{c}{\NN}{\NN'}-\change{\ell}{\NN}{\NN'} = 2-1=1$. 
Finally, $\change{\delta}{\NN}{\NN'}=\change{c}{\NN}{\NN'} - \change{\ell}{\NN}{\NN'} - \change{\operatorname{rk}}{\NN}{\NN'}=2-1-\change{\operatorname{rk}}{\NN}{\NN'}$ and so $\change{\delta}{\NN}{\NN'}$ equals $0$ or $1$.

\uline{Case 2}: $C$ is a new complex and $D$ is not ($C \notin \CC_\NN$, $D \in \CC_\NN$), or vice-versa. In this case, $\change{c}{\NN}{\NN'}=1$, and adding the reaction $C \to D$ (or $C \leftrightarrows D$) simply enlarges one linkage class (and no linkage classes are joined or created).  Hence, $\change{\ell}{\NN}{\NN'} = 0$
and so $\change{c}{\NN}{\NN'}-\change{\ell}{\NN}{\NN'} = 1-0=1$.  Finally, 
$\change{\delta}{\NN}{\NN'}=1-0-\change{\operatorname{rk}}{\NN}{\NN'}$ and so $\change{\delta}{\NN}{\NN'}$ equals $0$ or $1$.
        
\uline{Case 3}: $C, D \in \CC_\NN$, but $C$ and $D$ are in distinct linkage classes of $\NN$.
In this case, $\change{c}{\NN}{\NN'}=0$.  Also, adding the reaction $C\to D$ (or $C \leftrightarrows D$) joins two linkage classes, so $\change{\ell}{\NN}{\NN'} = -1$.  Hence, 
$\change{c}{\NN}{\NN'}-\change{\ell}{\NN}{\NN'} = 0-(-1)=1$.
Finally, we have 
$\change{\delta}{\NN}{\NN'}=0-(-1)-\change{\operatorname{rk}}{\NN}{\NN'}$ and so $\change{\delta}{\NN}{\NN'}$ equals $0$ or $1$.

\uline{Case 4}: $C, D \in \CC_\NN$, and $C$ and $D$ are in the same linkage class of $\NN$. 
We have $\change{c}{\NN}{\NN'}=0$ and $\change{\ell}{\NN}{\NN'} =0$.  Next, $\change{\operatorname{rk}}{\NN}{\NN'}=0$ by Lemma~\ref{E5' cycles lemma}.  Hence, $\change{\delta}{\NN}{\NN'}=\change{c}{\NN}{\NN'} - \change{\ell}{\NN}{\NN'} - \change{\operatorname{rk}}{\NN}{\NN'}=0-0-0=0$.
\end{proof}

\begin{remark} \label{rem:cappelletti}
Lemma~\ref{lem:add-reactions} implies that operations that consist of adding reactions (e.g., E1, E2, and E5) never decrease the deficiency.  
Another instance of such an operation is when some (but possibly not all) inflow reactions $0 \to X_i$ and outflow reactions $X_i \to 0$ are added. Results on which such operations preserve nondegenerate multistationarity were proven recently by Cappelletti, Feliu, and Wiuf~\cite[Theorem 8]{add-flow}.  However, it is unknown whether their operations preserve (nondegenerate) periodic orbits.  We conjecture that they do.
\end{remark}

The following example motivates Lemma \ref{E34 nondecreasing lem}.

\begin{ex}[E3] \label{E3 splits}
Consider the following network $\NN$:
\begin{center}
    \begin{tikzcd}
        & C \arrow[dl, shift right] \arrow[dr, shift right] & \\
        A \arrow[ur,shift right] \arrow[rr,shift right]& & B\arrow[ul, shift right] \arrow[ll,shift right]
    \end{tikzcd}
\end{center}
The following network $\NN'$ is obtained from $\NN$ by applying an E3 operation (adding species $D$ to three reactions without changing the rank):
\begin{center}
	\begin{tikzpicture}[scale=1.3]
	\draw [->] (0,0) -- (1,0);
	\draw [->] (1.1,0.1) -- (0.6,.65);
	\draw [<-] (-0.1,0.1) -- (0.4,.65);
    \node [left] at (0,0) {$A$};
    \node [right] at (1,0) {$B$};
    \node [above] at (0.5,.7) {$C$};
	\draw [<-] (3,0) -- (4,0);
	\draw [<-] (4.1,0.1) -- (3.6,.65);
	\draw [->] (2.9,0.1) -- (3.4,.65);
    \node [left] at (3,0) {$A+D$};
    \node [right] at (4,0) {$B+D$};
    \node [above] at (3.5,.7) {$C+D$};
	\end{tikzpicture}
\end{center}  
Notice that the single linkage class of $\NN$ ``breaks'' into two linkage classes in $\NN'$.  This idea appears in the proof of the next lemma.
\end{ex}

\begin{lemma}[E3 and E4]\label{E34 nondecreasing lem} Consider a network $\NN$.
\begin{enumerate}
    \item  If $\NN'$ is a network obtained from~$\NN$ 
    by applying an E3  operation, then 
    $\change{\ell}{\NN}{\NN'} \leq \change{c}{\NN}{\NN'}$. 
    \item  If $\NN'$ is a network obtained from~$\NN$ 
    by applying an E4  operation, 
    then 
    $\change{\ell}{\NN}{\NN'} +1  \leq \change{c}{\NN}{\NN'}$. 
\end{enumerate}
\end{lemma}
\begin{proof} 
We begin by proving part~1. Consider 
the map  $\pi: \CC_{\NN'} \to  \CC_{\NN}$ 
that restricts a complex to only the ``old'' species, 
that is, $\pi$ is 
defined by 
$(a_1X_1 + a_2X_2 + \dots + a_{s_\NN}X_{s_\NN} + a_Y Y) \mapsto (a_1X_1 + a_2X_2 + \dots + a_{s_\NN}X_{s_\NN})$. 
By definition of E3, 
$\pi$ is surjective and gives rise to another map, a bijection $\widetilde \pi: \RR_{\NN'} \to  \RR_{\NN}$ defined by $(C\to D) \mapsto (\pi(C) \to \pi(D)) $.  In words, $\widetilde \pi$ sends a reaction to the reaction it was before E3 added the species $Y$ to some complexes.

It is now straightforward to see that each linkage class $L$ of $\NN$ corresponds to one or more linkage classes $L_1,L_2,\dots, L_{k+1}$ of $\NN'$ (see Example \ref{E3 splits}).  
Thus, we can analyze one linkage class of $\NN$ at a time.
Accordingly, assume that $L$ is a linkage class of $\NN$ that corresponds to linkage classes $L_1,L_2,\dots, L_{k+1}$ of $\NN'$ (where $k \geq 0$).  It suffices to show that 
the number of ``new'' complexes (i.e., those not in $\CC_\NN$) in the linkage classes $L_1,L_2,\dots, L_{k+1}$ is at least $k$. 

To this end, define a graph $G$, as follows: the vertices are the linkage classes $L_1,L_2,\dots, L_{k+1}$, and the (undirected) edge $(L_i,L_j)$ exists when there exist complexes $C\in L_i$ and $D \in L_j$ such that $\pi(C)=\pi(D)$.  It is straightforward to see that the graph $G$ is connected; this uses the fact that the linkage classes $L_1,L_2,\dots, L_{k+1}$ of $\NN'$ correspond to a single linkage class (namely, $L$) of $\NN$.  Hence, $G$ has a spanning tree $T$ with $k$ edges.

Next, we use the tree $T$ to define the edges of a graph $H$ with vertex set $\CC_{\NN'}$.  
The (undirected) edges of $H$ arise as follows:
for each of the $k$ edges of $T$, which we denote by some $(L_i,L_j)$, pick 
 complexes $C\in L_i$ and $D \in L_j$ such that $\pi(C)=\pi(D)$, and then define $(C,D)$ to be an edge of $H$ (note that $C \neq D$, as 
 $C$ and $D$ are in distinct linkage classes $L_i$ and $L_j$, respectively).  It follows that $H$ (like $T$) has no cycles, and so is a forest with $k$ edges.  
Also, all vertices $C$ in a fixed (but arbitrary) connected component of $H$ correspond to the same ``old'' complex (that is, $\pi(C)$ is the same for all such vertices).  Hence, each connected component of $H$ 
contains at most one ``old'' complex.
So, a component 
 of $H$
 with $\ell \le k$ edges --- and hence $\ell+1$ vertices since the component is a tree --- will involve at least $\ell$ ``new'' complexes.  
 Therefore,
 the vertices incident to the $k$ edges of $H$ include at least $k$ ``new'' complexes, which completes the proof of part~1.

Next, we prove part~2. By definition, every E4 operation is obtained by a generalized E3 operation -- in which the rank condition $\operatorname{rk}(\NN) =\operatorname{rk}(\NN')$ need not be satisfied -- followed by adding the reactions $0 \leftrightarrows Y$. Notice that the proof of part 1 does not use this rank condition. 
As a result, the proof of part 1 also applies to the generalized E3 operation. Let $\widetilde \NN$ denote the network after this generalized E3 operation (i.e., before adding  $0 \leftrightarrows Y$).  We consider two cases:

\uline{Case 1}: There exists a linkage class of $\widetilde \NN$ containing $0$ and $Y$ as complexes. 
In this case, $\change{\ell }{\NN}{\NN'} =\change{\ell}{\NN}{\widetilde \NN} $.  Next, we define $\widetilde H$ to be the graph obtained by adding the edge $(0,Y)$ to the graph $H$ (which we constructed in the proof we gave above for part~1), 
where $H$ comes from the linkage class of $\NN$ that contains $0$.  As $0$ and $Y$ are in the same linkage class of $\widetilde \NN$ (and hence in the same linkage class of $\NN'$), the edge $(0,Y)$ was not an edge of $H$ (by construction) and additionally is not part of a cycle in $\widetilde H$.  Moreover, $\pi(0)=\pi(Y)=0$.  So,
the edge
$(0,Y)$ (more precisely, the vertices incident to that edge) generates one more ``new'' complex than what was found in part~1. 
This fact yields the inequality here:
\begin{align*}
    \change{c}{\NN}{\NN'} 
    ~\geq~
    1+ \change{\ell}{\NN}{\widetilde \NN}
    ~=~
    1+ \change{\ell}{\NN}{\NN'}~,
\end{align*}
and the equality here was proven above.

\uline{Case 2}:
No linkage class of $\widetilde \NN$ contains both $0$ and $Y$. 
In this case, Lemma~\ref{lem:add-reactions} implies that $\change{c}{\widetilde \NN}{\NN'} - \change{\ell}{\widetilde \NN}{\NN'}=1$.  This equality and part~1 (applied to the pair $(\NN, \widetilde \NN)$) together imply the desired inequality, as follows:
\begin{align*} 
    \change{\ell}{\NN}{\NN'}
    ~&=~
    \change{\ell}{\NN}{\widetilde \NN} + \change{\ell}{\widetilde \NN}{\NN'}  \\
    ~& \leq ~
    \change{c}{\NN}{\widetilde \NN} + \change{c}{\widetilde \NN}{\NN'} - 1
    \\
    ~&=~\change{c}{ \NN}{\NN'} - 1~.\qedhere
\end{align*}
\end{proof}

\begin{remark} \label{rem:lem-does-not-require-rank}
    In the proof of part~1 of Lemma~\ref{E34 nondecreasing lem}, 
    the rank condition of operation E3 is not used.  Hence, that part of the lemma allows for versions of E3 without the rank condition.
\end{remark}

Our final lemma 
pertains to block triangular matrices, which we later apply to stoichiometric matrices that satisfy the rank-$m$ condition.

\begin{lemma} \label{lem:block-upper-triangular}
    Consider a block upper-triangular matrix of size $(p+q) \times (u+v)$ with entries in $\mathbb{R}$:
\begin{align*} 
M
~=~
\left(
    \begin{array}{c|c}
        A & B \\
        \hline
        {\bf 0} & C
    \end{array}
\right)~,
\end{align*}
where ${\bf 0}$ represents a $(q \times u)$ zero matrix,
and
$A$, $B$, and $C$ have size $(p \times u)$, $(p \times v)$, and $(q \times v)$, respectively.  If  
$\operatorname{rk}(C)=v$, then $\operatorname{rk}(M) = \operatorname{rk}(A) + \operatorname{rk}(C)$.
\end{lemma}
\begin{proof}
    Since $\operatorname{rk}(C)=v$, we perform row operations on $M$ so that the entries of submatrix $B$ become zeros.
  The resulting matrix is a block diagonal matrix in which one block is $A$ and the other is a $q\times v$ submatrix of rank $v$. Hence, as operations do not affect the rank, we obtain $\operatorname{rk}(M) = \operatorname{rk}(A) + \operatorname{rk}(C)$.
\end{proof}

\subsection{E1}
Recall that the operation E1 adds a new reaction without changing the rank of the network (Definition~\ref{def:ops}).  
We also saw an example in which an E1 operation increases a network's deficiency by $1$ 
(Example~\ref{ex:E1-deficiency}).  The following result states that such an increase always occurs, except in the situation when the new reaction joins two complexes that already were in the same linkage class.

\begin{thm}[\hypertarget{E1_thm}{Deficiency after E1}]
\label{thm:E1}
    Let $\NN'$ be a network obtained from a network $\NN$ by applying an E1 operation.
    Let $C \to D$ denote the reaction added in this operation.
    If there exists a linkage class of $\NN$ that contains both $C$ and $D$ (in particular, the operation adds no new complexes), 
    then $$\change{\delta}{\NN}{\NN'} ~=~ 0.$$ 
    Otherwise,
    $\change{\delta}{\NN}{\NN'} = 1.$
\end{thm}

\begin{proof} By definition, the rank is unchanged by E1: $\change{\operatorname{rk}}{\NN}{\NN'} = 0$. 
We consider two cases.  First, assume that 
there exists a linkage class of $\NN$ that contains both $C$ and $D$.  In this case, $\change{c}{\NN}{\NN'}=0$ and $\change{\ell}{\NN}{\NN'} =0$, so $\change{\delta}{\NN}{\NN'}=\change{c}{\NN}{\NN'} - \change{\ell}{\NN}{\NN'} - \change{\operatorname{rk}}{\NN}{\NN'}=0-0-0=0$.  In the remaining case, Lemma~\ref{lem:add-reactions} implies that $\change{c}{\NN}{\NN'} - \change{\ell}{\NN}{\NN'}=1$ and so 
$\change{\delta}{\NN}{\NN'}= (\change{c}{\NN}{\NN'} - \change{\ell}{\NN}{\NN'}) - \change{\operatorname{rk}}{\NN}{\NN'}=1-0=1$.
\end{proof}

\begin{ex} \label{ex:E1-that-preserves-deficiency}
Theorem~\ref{thm:E1} implies that when an E1 operation makes a non-reversible reaction reversible, the deficiency is preserved.  For instance, the network $\NN = \{ A+B \to 2A\}$
transforms under such an E1 operation to $ \NN' = \{A+B \leftrightarrows 2A\}$, and both networks have deficiency $0$.
\end{ex}

\subsection{E2}
Recall that the operation E2 adds inflow-outflow reactions for all species 
(Definition~\ref{def:ops}).  To state our result on E2, we must give a name to the types of complexes added by E2, as follows.

\begin{defn} \label{def:special-complex}
    A complex is {\em at-most-unimolecular} (for short, {\em unimolecular}) if it is $0$ or $X_i$, for some $i$.  
\end{defn}

\begin{notation}[$m_\NN$ and $\widetilde \ell_\NN$] \label{notation:special-complex}
    Let $\NN$ be a network.
    \begin{enumerate}
        \item Let $m_\NN$ be the number of unimolecular complexes that are ``missing'' from $\NN$ (more precisely, the number of complexes of the form $0$ or $X_i$, with $i \in [s_\NN]$, that are not in $\CC_\NN$).
       \item Let $\widetilde \ell_\NN$ denote the number of linkage classes of $\NN$ that contain at least one unimolecular complex.
    \end{enumerate}
\end{notation}

\begin{ex}[Example~\ref{ex:E2}, continued] \label{ex:E2-part-2}
The network $\NN = \{A \to 2C,~2D \leftarrow C \to B \}$ has $m_\NN=2$ ``missing'' unimolecular complexes (namely, $0$ and $D$), and both linkage classes contain unimolecular complexes (so, $\widetilde \ell_\NN=2$).
\end{ex}

Our main result concerning E2 is as follows.
\begin{thm}[\label{E2_thm}\hypertarget{E2_thm}{Deficiency after E2}]
 \label{thm:E2}
If $\NN'$ is a network obtained from a network $\NN$ by applying the E2 operation, then $\change{\delta}{\NN}{\NN'} \ge 0$ and
    $$\change{\delta}{\NN}{\NN'} ~=~ 
    \operatorname{rk}(\NN)  - s_\NN + m_\NN + \widetilde \ell_\NN- 1~,$$
where the notation is as in Table~\ref{tab:notation} and Notation~\ref{notation:special-complex}.    
\end{thm}

\begin{proof} 
    First, $\change{\delta}{\NN}{\NN'} \geq 0$ follows from repeated application of
    Lemma~\ref{lem:add-reactions}.
    Next, $\change{c}{\NN}{\NN'} = m_{\NN}$, because the complexes that E2 adds are precisely the $m_{\NN}$ ``missing'' complexes. 
       Also, $\operatorname{rk}(\NN') = s_\NN$, 
       as $\NN'$ contains all inflow-outflows; hence, 
    $\change{\operatorname{rk}}{\NN}{\NN'} = s_\NN - \operatorname{rk}(\NN)$. 
    
     Now we claim that $\change{\ell}{\NN}{\NN'} = 1- \widetilde \ell_\NN$.  To prove this claim, we consider two cases.  First, if $\widetilde{\ell}_{\NN} =0$, then E2 adds exactly one new linkage class (consisting of all unimolecular complexes), so 
    $\change{\ell}{\NN}{\NN'}= 1=1 - \widetilde{\ell}_{\NN}$.  On the other hand, if $\widetilde{\ell}_{\NN} \geq 1$, then  E2 connects the $\widetilde{\ell}_{\NN}$ linkage classes containing unimolecular complexes into one linkage class, and so $\change{\ell}{\NN}{\NN'} =
    - \widetilde{\ell}_{\NN} +1$.

We now can compute $\change{\delta}{\NN}{\NN'}$:
    \begin{align} \label{eq:change-E2}
    \notag
        \change{\delta}{\NN}{\NN'} &~=~ \change{c}{\NN}{\NN'} - \change{\ell}{\NN}{\NN'} - \change{\operatorname{rk}}{\NN}{\NN'} \\
               &~=~ m_{\NN} - (1-\widetilde \ell_\NN ) - (s_{\NN} - \operatorname{rk}(\NN)) ~.
    \end{align}
Finally, rearranging the terms in~\eqref{eq:change-E2} yields exactly the desired expression for $\change{\delta}{\NN}{\NN'}$.
\end{proof}

\begin{ex}[Example~\ref{ex:E2-part-2}, continued] \label{ex:E2-part-3}
Recall that for the network $\NN = \{A \to 2C,~2D \leftarrow C \to B \}$ and the network $\NN'$ obtained from the E2 operation, we have $\change{\delta}{\NN}{\NN'} =2$.  This is consistent with Theorem~\ref{thm:E2}, as $\operatorname{rk}(\NN)  - s_\NN + m_\NN + \widetilde \ell_\NN- 1= 3 - 4 + 2 + 2 - 1 = 2$.
\end{ex}

\subsection{E3}
Recall that operation E3 adds a new species without changing the network's rank (Definition~\ref{def:ops}).  
Also, we presented an example in which an E3 operation leaves unchanged both the deficiency and the cyclomatic number of a network (Example~\ref{ex:E3}).  
The following result shows that when the deficiency does change, it must increase and, moreover, its change is exactly opposite the change in the cyclomatic number.

\begin{thm}[\hypertarget{E3_thm}{Deficiency after E3}]
 \label{thm:E3}
 If $\NN'$ is a network obtained from a network~$\NN$ by applying an E3 operation, then $\change{\delta}{\NN}{\NN'} \ge 0$ and
    $$\change{\delta}{\NN}{\NN'} ~=~ 
    - \change{\operatorname{cyc}}{\NN}{\NN'} ~.$$
\end{thm}
\begin{proof}
By definition, E3 preserves the rank of the network ($\change{\operatorname{rk}}{\NN}{\NN'} = 0$) 
and adds species to reactions, but does not change the number of reactions ($\change{r}{\NN}{\NN'} = 0$).  We compute the change in deficiency:
    \begin{align} \label{eq:proof-E3}
        \notag
        \change{\delta}{\NN}{\NN'} 
        &~=~ \change{c}{\NN}{\NN'} - \change{\ell}{\NN}{\NN'} - \change{\operatorname{rk}}{\NN}{\NN'}  \\
        &~=~ \change{c}{\NN}{\NN'} - \change{\ell}{\NN}{\NN'} \\
        \notag
        &~=~ - \left( \change{r}{\NN}{\NN'} - \change{c}{\NN}{\NN'} + \change{\ell}{\NN}{\NN'} \right)  \\
        \notag
        &~=~ -\change{\operatorname{cyc}}{\NN}{\NN'}.
    \end{align}
Finally, the inequality $\change{\delta}{\NN}{\NN'} \ge 0$
follows from equation~\eqref{eq:proof-E3} and Lemma~\ref{E34 nondecreasing lem} (part~1).
\end{proof}

\begin{ex}[Example \ref{E3 splits} continued] 
We revisit the networks in Example~\ref{E3 splits}, where $\NN'$ is obtained from $\NN$ by an E3 operation. The deficiencies  are
$\delta(\NN)=3-1-2=0$ and
$\delta(\NN')=6-2-2=2$, and 
the cyclomatic numbers
are
$\operatorname{cyc}(\NN)=6-3+1=4$
and $\operatorname{cyc}(\NN')=6-6+2=2$.
Consistent with Theorem~\ref{thm:E3}, we have $\change{\delta}{\NN}{\NN'} = 2 = - \change{\operatorname{cyc}}{\NN}{\NN'}$.
\end{ex}

\subsection{E4}
Recall that the operation E4 adds a new species $Y$ into some reactions and also adds the reactions $0 \leftrightarrows Y$ (Definition~\ref{def:ops}).  
Also, in Example~\ref{ex:E4},  we saw an instance of an E4 operation that increases the deficiency by $1$ (that is, $\change{\delta}{\NN}{\NN'} =  1$) but does not change the cyclomatic number ($\change{\operatorname{cyc}}{\NN}{\NN'} =  0$).
The main result of this subsection states that, in general, the operation E4 yields the equality
$\change{\delta}{\NN}{\NN'} + \change{\operatorname{cyc}}{\NN}{\NN'} =  1$ (Theorem~\ref{thm:E4}).  We need the following lemma.

\begin{lemma}\label{E4_lem}
 If $\NN'$ is a network obtained from a network~$\NN$ by applying an E4 operation, then  
$ \change{r}{\NN}{\NN'} = 2$
and 
 $\change{\operatorname{rk}}{\NN}{\NN'} =  1$.
\end{lemma}

\begin{proof} 
The operation E4 adds a new species $Y$ (into existing reactions) and additionally adds the pair of reversible reactions $0 \leftrightarrows Y$ (so, the change in the number of reactions is $ \change{r}{\NN}{\NN'} = 2$).  The stoichiometric matrix of $\NN'$ therefore has the following block lower-triangular form:
\begin{align} \label{eq:stoic-matrix-E4}
\Gamma_{\NN'}
~=~
\left(
	\begin{array}{ccc|cc}
	&&&    0 &  0 \\	
	& \Gamma_{\NN} & & \vdots & \vdots   \\
	&& & 0 & 0 \\	
    \hline
    * & \dots & * & 1 & -1 \\
	\end{array}
\right)~,
\end{align}
where the last row corresponds to species $Y$ and the last two columns correspond to the reactions $0 \to Y$ and $0 \leftarrow Y$, respectively.  
By construction, the first $s_{\NN}$ rows of $\Gamma_{\NN'}$, in~\eqref{eq:stoic-matrix-E4}, span a $\operatorname{rk}({\NN})$-dimensional subspace, which we denote by $S$.  Next, the last row of $\Gamma_{\NN'}$ is readily seen to not be in $S$.  Hence, $\operatorname{rk}(\NN')= \operatorname{rk}(\NN) +1$, that is, $\change{\operatorname{rk}}{\NN}{\NN'} =  1$.
\end{proof}

\begin{thm}[Deficiency after E4]\label{E4 Theorem}\label{thm:E4}
 If $\NN'$ is a network obtained from a network~$\NN$ by applying an E4 operation, then $\change{\delta}{\NN}{\NN'} \ge 0$ and
    $$\change{\delta}{\NN}{\NN'} ~=~ 
        -\change{\operatorname{cyc}}{\NN}{\NN'} + 1.$$
\end{thm}

\begin{proof}
Using Lemma~\ref{E4_lem} and the definition of cyclomatic number, we compute as follows:
    \begin{align} \label{eq:proof-E4}
    \notag
    \change{\delta}{\NN}{\NN'} &~=~ \change{c}{\NN}{\NN'} - \change{\ell}{\NN}{\NN'} - \change{\operatorname{rk}}{\NN}{\NN'}  \\
        &~=~ \change{c}{\NN}{\NN'} - \change{\ell}{\NN}{\NN'} - 1\\
    \notag
        &~=~ -\change{\operatorname{cyc}}{\NN}{\NN'} + \change{r}{\NN}{\NN'} - 1
        \\
    \notag
        &~=~ -\change{\operatorname{cyc}}{\NN}{\NN'} + 1.
    \end{align}
Finally, 
$\change{\delta}{\NN}{\NN'} \geq 0$ follows from equation~\eqref{eq:proof-E4} and Lemma~\ref{E34 nondecreasing lem} (part~2).
\end{proof}

\begin{ex} \label{ex:after-E4-theorem}
In the following networks, which we denote by $\NN$ and $\NN'$, respectively, the network $\NN'$ is obtained from $\NN$ by performing an E4 operation:
\begin{center}
    \includegraphics[width=12cm]{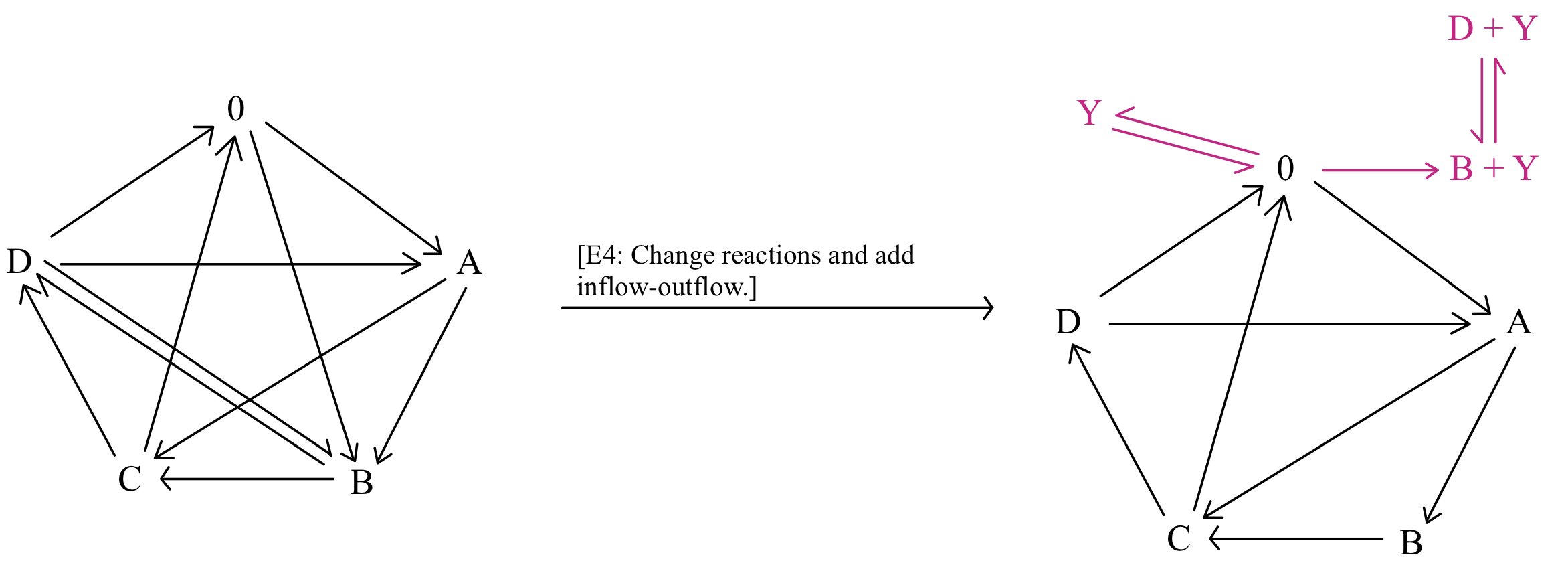}
\end{center}
The deficiencies are $\delta(\NN) = 5 - 1 - 4 = 0 \text{ and } \delta(\NN') = 8 - 1 - 5 = 2$, while the cyclomatic numbers are
$\operatorname{cyc}(\NN) = 11 - 5 + 1 = 7$ and $\operatorname{cyc}(\NN') = 13 - 8 + 1 = 6$. 
Observe that 
$\change{\delta}{\NN}{\NN'} +
\change{\operatorname{cyc}}{\NN'}{\NN} = 2+(-1)=1$, which is consistent with Theorem~\ref{thm:E4}.
\end{ex}

\subsection{E5}
Recall that the operation E5 adds $m$ new reversible reactions and at least $m$ new species to those reactions, such that the rank-$m$ condition holds (Definition~\ref{def:ops}).  
Also, an instance of this operation was seen to preserve the deficiency (Example~\ref{ex:E5}).  We now prove that this holds in general.

\begin{thm}
 \label{thm:E5}
 If $\NN'$ is a network obtained from a network~$\NN$ by applying an E5 operation, then $\change{\delta}{\NN}{\NN'} = 0$. 
\end{thm}

\begin{proof}
Denote the $m$ new reactions (involving $m+i$ new species) of the E5 operation as follows:
\begin{align}
    \label{eq:reactions-E5}
C(1) \leftrightarrows D(1)~, \quad  \dots, \quad 
C(m) \leftrightarrows D(m)~.
\end{align}
The stoichiometric matrix for $\NN'$ (where only one reaction vector is included for each added reversible reaction, as in Remark~\ref{rem:reversible-reactions-stoic-matrix}) is as follows:
\begin{align} \label{eq:matrix-E5-proof}
\Gamma_{\NN'}
~=~
\left(
	\begin{array}{ccc|ccc}
	&&&    (D(1)-C(1))_1 & \dots & (D(m)-C(m))_1 \\	
	& \Gamma_{\NN} & & \vdots &  & \vdots  \\
	&& &  & & \\	
    \cline{1-3} 
	0 & \dots &  0 &   &    &  \\
	\vdots & \ddots &  \vdots &	\vdots &  & \vdots \\
    0 & \dots & 0 & (D(1)-C(1))_{s_{\NN}+ m+i} & \dots & (D(m)-C(m))_{s_{\NN}+ m+i} \\
	\end{array}
\right)~.
\end{align}
The rank-$m$ condition for E5, plus the block structure of $\Gamma_{\NN'}$ in~\eqref{eq:matrix-E5-proof}, together satisfy the hypotheses of Lemma~\ref{lem:block-upper-triangular}, and so we conclude the following:
\begin{align} \label{eq:rank-change-E5}
    \change{\operatorname{rk}}{\NN}{\NN'}~=~m~.
\end{align}

Our next aim is to compute $\change{c}{\NN}{\NN'}- \change{\ell}{\NN}{\NN'}$ using Lemma~\ref{lem:add-reactions}.  To this end, for $j \in [m]$, let $\NN(j)$ denote the network obtained from $\NN$ by adding to $\NN$ the first $j$ reactions of~\eqref{eq:reactions-E5} (so, $\NN(m)=\NN'$).  For convenience, let $\NN(0):=\NN$.

Fix $j \in [m]$.  We claim that
no linkage class of $\NN(j-1)$ contains both complexes $C(j)$ and $D(j)$.  Indeed, if they were in the same linkage class of $\NN(j-1)$, then the column vector $D(j)-C(j)$, in~$\Gamma_{\NN'}$ as shown in~\eqref{eq:matrix-E5-proof}, would be in the column space of the submatrix to the left of that vector (by Lemma~\ref{E5' cycles lemma}), and it is straightforward to show that this would violate~\eqref{eq:rank-change-E5}.  

We conclude that Lemma~\ref{lem:add-reactions} applies and so $\change{c}{\NN(j)}{\NN(j-1)} - \change{\ell}{\NN(j)}{\NN(j-1)}=1$.  This equality yields:
\begin{align} \label{eq:total-change-c-ell-E5}
\change{c}{\NN}{\NN'} - \change{\ell}{\NN}{\NN'} ~&=~
\sum_{j=1}^m \left( \change{c}{\NN(j)}{\NN(j-1)}- \change{\ell}{\NN(j)}{\NN(j-1)} \right)
~=~
\sum_{j=1}^m 1~=~ m~.
\end{align}
Hence, by~\eqref{eq:rank-change-E5}
 and~\eqref{eq:total-change-c-ell-E5}, 
 $\change{\delta}{\NN}{\NN'}= {\big (}\change{c}{\NN}{\NN'} - \change{\ell}{\NN}{\NN'} {\big )} -\change{\operatorname{rk}}{\NN}{\NN'}=m-m=0$.
\end{proof}

\subsection{E6}
Recall that the operation E6 splits $m$ reactions, and inserts 
at least $m$
species in the new complexes 
(Definition~\ref{def:ops}).  We saw an example in which this operation preserves the deficiency (Example~\ref{ex:E6}), and now we show that this holds in general.

\begin{thm}\label{E6 Theorem}
 \label{thm:E6}
 If $\NN'$ is a network obtained from a network~$\NN$ by applying an E6 operation, then $$\change{\delta}{\NN}{\NN'} ~=~ 0~.$$
\end{thm}

\begin{proof}
Let $m$ denote the number of split reactions, and let $m+i$ denote the number of new species (as in Definition~\ref{def:ops}).  Recall that $\change{\delta}{\NN}{\NN'} =
    \change{c}{\NN}{\NN'} - \change{\ell}{\NN}{\NN'} - \change{\operatorname{rk}}{\NN}{\NN'}$.  It therefore suffices to prove the following: 
    (1) $\change{c}{\NN}{\NN'} = m$, 
    (2) $\change{\ell}{\NN}{\NN'}=0$, and 
    (3) $\change{\operatorname{rk}}{\NN}{\NN'} = m$.  

We begin by proving~(3).  We first prove the following claim:

\uline{Claim 1}: The following
$(s_{\NN} +m+i) \times (r_\NN + r_{\NN'})$ matrix has the same column space as $\Gamma_{\NN'}$ (the stoichiometric matrix of $\NN'$):
\begin{align} \label{eq:bigger-matrix}
M
~=~
\left(
	\begin{array}{ccc|ccc}
	&&&     &  &  \\	
	& \Gamma_{\NN} & &  &  &    \\
	&& &  & \Gamma_{\NN'} &  \\	
    \cline{1-3} 
	0 & \dots &  0 &   &    &  \\
	\vdots & \ddots &   &	 &  &  \\
    0 & \dots & 0 &  &  &  \\
	\end{array}
\right)~.
\end{align}

To verify this claim, first observe that the matrix $M$, in~\eqref{eq:bigger-matrix}, is obtained from $\Gamma_{\NN'}$ by 
adjoining (on the left) the reaction vectors $D-C$ (viewed in $\mathbb{R}^{s_{\NN}+m+i}$ rather than in $\mathbb{R}^{s_{\NN}}$) for each reaction $C \to D$ of $\NN$.  For each such reaction, there are two possibilities.  First, if $C\to D$ is {\em not} split by the E6 operation, then $D-C$ is also a column of $\Gamma_{\NN'}$ and so is in the column space of $\Gamma_{\NN'}$.  On the other hand, if $C\to D$ is split and becomes $C \to Z \to D$, where $Z$ is a new complex, then the reaction vectors $Z-C$ and $D-Z$ are columns of $\Gamma_{\NN'}$ and hence their sum $(Z-C)+(D-Z)=D-C$ is in the column space of $\Gamma_{\NN'}$.  We conclude that Claim~1 is true.

Next, Claim~1 and the block structure of the matrix $M$, in~\eqref{eq:bigger-matrix}, together satisfy the hypotheses of Lemma~\ref{lem:block-upper-triangular}. We conclude that
$\operatorname{rk}(\NN') $ is the sum of $\operatorname{rk}(\NN) $ and the rank of the lower-right submatrix of $M$.  This submatrix, by definition of E6, has rank $m$.  Thus, $\change{\operatorname{rk}}{\NN}{\NN'} = m$, and so part~(3) holds.

Next, to prove~(1), it suffices to show the following claim:

\uline{Claim 2}: The complexes added in the $m$ split reactions are distinct, and each such complex involves at least 1 new species.

To prove Claim~2, we begin by denoting the $m$ reactions -- after they are split by E6 -- as follows:
\[
C(1) \to Z(1) \to D(1)~, \quad  \dots, \quad 
C(m) \to Z(m) \to D(m)~.
\]
We must show that each complex $Z(j)$ is distinct and involves new species.

Let $\pi: \N_0^{s_{\NN}+m+i} \to  \N_0^{m+i} $ denote the projection onto the coordinates for the $m+i$ new species.  With this notation, we see that in the submatrix formed by the rows of $\Gamma_{\NN'}$ corresponding to the new species -- we denote this submatrix by $\Lambda$ -- the columns are $\pi(v_j)$, where $v_j$ is the reaction vector of the $j$-th reaction of $\NN'$.  If such a reaction is {\em not} split by E6, then $\pi(v_i)=0$ (the reaction does not involve new species).  On the other hand, for reactions that are split, the corresponding columns come in pairs, namely,
$\pi(Z(j)-C(j))=\pi(Z(j))$ (as $C(j)$ is a complex of $\NN$ and hence does not involve new species) and 
$\pi(D(j)-Z(j))=-\pi(Z(j))$.  

Thus, the set 
$B=\{\pi(Z(j)) \mid j\in [m]\}$ spans
the column space of $\Lambda$, which, by definition of E6 (and the fact that column rank equals row rank), has dimension $m$.
Hence, $B$ is a basis of the column space and so
the $\pi(Z(j))$, for $j\in [m]$, are nonzero and distinct.  Therefore, the complexes $Z(j)$ must involve new species and also be distinct, which completes the proof of Claim~2.

Only part~(2) remains to be proven.  Claim~2 implies that when reactions are split in the E6 operation, linkage classes are enlarged, but not joined or created.  Hence, $\change{\ell}{\NN}{\NN'}=0$.
\end{proof}

\section{Discussion} \label{sec:discussion}
In this work, we established a connection between two research streams pertaining to chemical reaction networks, namely, deficiency theory and the theory of ``lifting'' results.
Specifically, we proved that the 
deficiency never decreases when any of the six network operations E1--E6 are performed, and, moreover, we characterized the numerical difference in deficiency after performing each of these operations (Theorem \ref{thm:main-result-summary}). 
This means that chemical reaction networks may be arbitrarily enlarged using any of these operations, and the resulting networks preserve important dynamical properties --  nondegenerate multistationarity and periodic orbits~\cite{Banaji} 
and bifurcations~\cite{banaji2023inheritance}
-- 
with predictable changes in deficiency. In particular, only E5 and E6 always  leave the deficiency unchanged.

Additionally, we posed two conjectures.  First, we conjectured that the operation E5$'$, which generalizes E5, does not preserve nondegenerate periodic orbits (Section~\ref{sec:general-rank-condition}).
On the other hand, we conjectured that 
a network operation that generalizes E2, which was introduced recently by Cappelletti, Feliu, and Wiuf, 
does preserve nondegenerate periodic orbits (Remark~\ref{rem:cappelletti}).

Going forward, the theory of ``lifting'' results is an active research area. We therefore anticipate wanting to prove results pertaining to deficiency that are analogous to the results in this work, 
as new ``lifting'' results are proven. 
Our proofs for the operations E1--E6 lay the groundwork for such future investigations.


\section*{Acknowledgements}
{\small
Awildo Gutierrez, Elijah Leake, and Caelyn Rivas--Sobie initiated this
research in the 2022 MSRI-UP REU, 
which is supported by the NSF (DMS-2149642) and the Alfred P. Sloan Foundation, in which Jordy Lopez Garcia and Anne Shiu were mentors. 
The authors are grateful to the Simons Laufer Mathematical Sciences Institute (SLMath, formerly MSRI) for providing a magnificent working environment and to Scribble Together for supporting our remote collaboration.
AS was supported by the NSF (DMS-1752672).
The authors thank 
Federico Ardila
for many helpful discussions -- especially pertaining to cyclomatic numbers -- and 
for his leadership as on-site director of MSRI-UP. 
The authors also thank Saber Ahmed, Sally Cockburn, 
and Natasha Crepeau
for helpful comments on an earlier draft of this work, and Chun-Hung Liu for his perspective on cyclomatic numbers. Finally, we would like to thank the anonymous referees for their detailed suggestions.

\end{document}